\documentclass[11pt,a4paper]{amsart}
\usepackage[T1]{fontenc}
\usepackage{lmodern,microtype,amsmath,amssymb,amsthm,mathtools}
\usepackage[margin=30mm]{geometry}
\usepackage[colorlinks=true,linkcolor=blue!50!black,citecolor=blue!50!black,urlcolor=blue!50!black]{hyperref}
\usepackage{xcolor,amsrefs}
\newtheorem{Theorem}{Theorem}[section]
\newtheorem{Lemma}[Theorem]{Lemma}
\newtheorem{Proposition}[Theorem]{Proposition}
\newtheorem{ExternalTheorem}[Theorem]{Theorem}
\theoremstyle{remark}
\newtheorem{Remark}[Theorem]{Remark}
\numberwithin{equation}{section}
\newcommand{\ind}{\mathbf1}
\DeclareMathOperator{\Aut}{Aut}
\title[Critical limits for face-constrained labelled maps]{Critical limits of generating functions for face-constrained labelled planar maps}
\author{Zhenyao Sun}
\address{School of Mathematics and Statistics, Beijing Institute of Technology, Beijing, China}
\email{zhenyao.sun@gmail.com}
\date{}
\begin{document}
\begin{abstract}
We study connected planar maps with a fixed number of numbered faces, whose vertices carry integer labels such that the minimum label along the boundary of each face is one. Each face has its own positive weight, and an edge shared by two faces receives the geometric mean of their weights. For any fixed number $k\ge3$ of faces, we prove that the suitably rescaled generating function converges locally uniformly as the face weights approach the critical value $1/12$, possibly at different positive rates. The limit is a finite, continuous sum of integrals over cubic schemes. Each scheme edge carries a Green kernel for a singular differential operator on the positive half-line, and each face contributes one insertion of the derivative of the tree scaling profile. The main difficulty is obtaining uniform estimates for small vertex labels, where the minimum-label constraints of different faces can interact. A positive corner-marking expansion reduces this problem to establishing uniform bounds on the contributions of marked configurations, both for small and large vertex labels. These estimates prove integrability of the surviving diagrams and show that coincident marks and noncubic schemes contribute only lower-order terms.
\end{abstract}
\maketitle
\enlargethispage{3pt}

\section{Introduction}
\subsection{From finite schemes to simultaneous face minima}
Fixing the number of faces of a planar map does not bound its size. One may attach arbitrarily large plane trees or subdivide edges into arbitrarily long chains without adding a face. Removing these two sources of growth leaves a finite collection of reduced maps, usually called schemes. More precisely, after pruning vertices of degree one and suppressing vertices of degree two, a scheme with $k\ge3$ faces has at most $2k-4$ vertices and $3k-6$ edges. The equalities hold precisely when every vertex has degree three. This decomposition separates the finitely many possible schemes from the arbitrarily large tree and chain decorations that give rise to the critical behaviour of the generating functions.

The vertex labels introduce an additional scale that becomes relevant as the generating functions approach their critical point. We consider labels whose increments along an edge belong to $\{-1,0,1\}$. Such labels are natural in the enumeration of maps with distinguished vertices: bijections convert distance information into label restrictions on trees or on maps with several faces. Miermont's construction \cite{miermont}*{Section~2} provides this connection for an arbitrary number of distinguished vertices and in arbitrary genus. In this paper the labelled maps themselves are the objects being counted, and we keep the planar topology throughout.

The simplest local counting problem is a plane tree planted at a corner with root label $l$, all labels at least one, and weight $g$ per edge. Its generating function $R_l(g)$ has an explicit expression, recalled in Section~\ref{sec:trees}. Far from the lower label boundary, it tends to the solution of
\[
 R_\infty=1+3gR_\infty^2,
 \qquad R_\infty(g)=\frac{1-\sqrt{1-12g}}{6g}.
\]
Thus the bulk tree series becomes singular at $g=1/12$. As $g$ approaches the critical value $1/12$, the constraint that all vertex labels are at least one influences the tree generating function $R_l(g)$ even when the root label $l$ is of order $(1-12g)^{-1/4}$. The critical asymptotics of the tree generating function $R_l(g)$ are described by a continuous scaling profile, known from the work of Bouttier and Guitter \cite{bouttierguitter}*{equations~(1.7)--(1.9)}. The derivative of this profile determines the asymptotic contribution of trees required to contain a vertex with label one.

These local formulas have already led to detailed results for maps. Bouttier and Guitter \cite{bouttierguitter} compute the three-point function of planar quadrangulations by decomposing labelled maps into trees, chains and branching pieces. Their Sections~3--4 supply the three-face scheme decomposition and the tree formulas used below. Miermont \cite{miermont}*{Section~3 and Theorem~5 in Section~4.3} develops the general fixed-face scheme decomposition and proves convergence of rescaled labelled-map measures, with cubic schemes surviving in the limit. His convergence theorem applies to continuous test functions satisfying a specified growth bound. Requiring the minimum vertex label along the boundary of each face to equal one imposes a discrete constraint that cannot be directly handled by Miermont's convergence theorem. Establishing the corresponding scaling limit therefore requires additional estimates for small vertex labels.

Beyond imposing minimum-label constraints on several faces simultaneously, we are also interested in allowing each face to have its own weight in the generating function. This multivariate formulation retains information that would be lost with a single edge-counting variable. In the two-face setting, Guitter \cite{guitter}*{Sections~2--4} studies exactly the minimum-one condition on both faces, assigning weights $g$ and $h$ to edges within a face and $\sqrt{gh}$ to shared edges. His analysis concerns two faces and exploits their cyclic skeleton. For three or more faces, the reduced skeleton branches, and distinct face constraints can meet at a common vertex or along a common chain.

Let $\mathcal Z_k(g_1,\ldots,g_k)$ denote the generating function for labelled planar maps with $k$ numbered faces, subject to the constraint that the minimum vertex label along the boundary of each face is one. Its precise definition is given below. We approach the critical point along
\[
 g_i=\frac{1-\lambda_i\varepsilon^4/36}{12},
 \qquad \lambda_i>0,\quad\varepsilon\downarrow0.
\]
We aim to understand the critical asymptotics of $\mathcal Z_k(g_1,\ldots,g_k)$ as $\varepsilon\to0$, uniformly for the positive parameters $\lambda_1,\ldots,\lambda_k$ in compact sets. In particular, we ask whether a suitable rescaling yields a nontrivial limit and whether this limit admits an explicit representation in terms of integrals over schemes.

Our main result identifies the rescaling factor as $\varepsilon^{4k-10}$ and proves locally uniform convergence to a finite sum of scheme integrals, each of which is finite and depends continuously on the positive parameters.

The scheme decomposition and the dominance of cubic schemes are known from the earlier work discussed above. Our main contribution is to establish uniform estimates that accommodate simultaneous minimum-label constraints on all faces and control the singular behaviour arising when multiple vertex labels become small simultaneously, for any fixed number $k\ge3$ of faces and positive parameters $\lambda_1,\ldots,\lambda_k$.

Our proof starts with a positive corner-marking expansion, which allows us to separate the candidate leading contributions from the errors caused by multiple marks. To explain this expansion, consider the minimum-label constraint on a single face. After summing over the trees attached to the corners of a given face, the constraint on the vertex labels requires at least one of these trees to contain a vertex with label one. Choosing one such corner gives a positive sum, but overcounts configurations with several trees touching label one. We first study the leading contributions, corresponding to cubic schemes with one marked corner per face, where the marked corners belong to distinct internal chain vertices. Under the critical scaling $x=\varepsilon l$, the rescaled weight of a marked corner behaves like $x^{-3}$ as $x\to0$. The main difficulty is to prove that the resulting integrals remain finite when several rescaled vertex labels approach zero simultaneously.

To overcome this difficulty, we establish a combinatorial inequality that bounds the number of faces incident to any subset of vertices of a planar map. Combined with the estimates for the tree and chain weights, this inequality gives sufficient decay for the total integral contribution on each small-label scale. A dyadic decomposition then allows us to sum these bounds over all possible scales and establish the integrability of the limiting expressions.

It remains to show that all other configurations give lower-order contributions. These include configurations with additional marked corners, which arise from overcounting in the positive marking expansion, and configurations in which marks from different faces share a vertex or a mark lies at an original scheme vertex. We also need to control noncubic schemes, whose contributions have a different scaling order.

To estimate these terms, we retain the discrete lower cutoff at label one and combine bounds for small and large vertex labels with power-counting arguments that take into account the combinatorial structure of each configuration. We then show that all these contributions vanish after rescaling, so that only the integrals associated with cubic schemes and marks at distinct internal chain vertices remain in the limit. We use the known critical scaling profile of the tree generating function and establish the convergence of the chain weights through an operator-theoretic approach, without requiring explicit formulas for chains whose two adjacent faces have different weight parameters.

\subsection{Generating functions, limiting integrals, and the main theorem}
We study labelled planar maps with a fixed number of numbered faces and a minimum-label condition on each face. Their positive generating series is the discrete object whose critical limit we will identify.

A planar map is a finite connected graph embedded cellularly in the oriented sphere, considered up to orientation-preserving homeomorphism. Loops and multiple edges are allowed, and the cyclic order of edge incidences at every vertex is part of the map.

Let $\mathfrak M_k$ be the set of isomorphism classes of connected planar maps with $k$ numbered faces and integer vertex labels differing by at most one along each edge, such that the minimum of the labels incident to each face is exactly one. Each map $M$ has weight $1/|\Aut(M)|$, where automorphisms preserve the orientation, the numbered faces and the vertex labels. An edge whose two sides have face labels $i,j$ has weight $\sqrt{g_i g_j}$; an $ii$ edge has weight $g_i$. If $a_i(M)=\deg(f_i)/2$ is half the boundary degree of face $i$, with repeated sides counted, the weight is $\prod_i g_i^{a_i(M)}$ and $\sum_i a_i(M)$ is the number of edges. Thus the generating function is
\[
 \mathcal Z_k(g_1,\ldots,g_k)
 =\sum_{M\in\mathfrak M_k}\frac1{|\Aut(M)|}\prod_{i=1}^k g_i^{a_i(M)}.
\]
To interpret the series formally, set $g_i=z w_i$ with fixed positive $w_i$ and expand in $z$. For numerical positive weights, $\mathcal Z_k(g)$ denotes the sum of these nonnegative map weights, equivalently the increasing limit of the sums over maps with at most $N$ edges as $N\to\infty$. This defines a value in $[0,+\infty]$ without any summation-order convention; finiteness at the parameters used below is part of the analytic argument.

The critical scaling used in the theorem and throughout the paper is
\begin{equation}\label{eq:critical-scaling}
 g_i(\varepsilon)=\frac{1-\lambda_i\varepsilon^4/36}{12},\qquad
 \eta_i=\lambda_i/4,
\end{equation}
with positive $\lambda$ fixed and $\varepsilon\downarrow0$.

The limit is expressed as a finite sum of integrals associated with marked cubic schemes. We define the half-line Green kernels and the possible face-mark insertions before stating that limit.

For $\eta>0$ let $a=(4\eta)^{1/4}$ and define
\begin{equation}\label{eq:potential}
 u_\eta(x)=\frac{a^2}{2}\left(1+3\operatorname{csch}^2(ax/2)\right),\qquad x>0.
\end{equation}
Then $u_\eta''=u_\eta^2-\eta$, $u_\eta(\infty)=\sqrt\eta$, and
\begin{equation}\label{eq:potential-expansion}
 u_\eta(x)=\frac6{x^2}+\frac{\eta x^2}{10}-\frac{\eta^{3/2}x^4}{126}+O(x^6),\qquad x\downarrow0.
\end{equation}
For a positive parameter vector $\eta$, write $u_i=u_{\eta_i}$ and $p_i=-u_i'>0$. Let $K_{ij}$ be the positive Green kernel of the Friedrichs operator
\[
 H_{ij}=-\frac{d^2}{dx^2}+u_i+u_j
\]
on $(0,\infty)$. The Friedrichs realization is the self-adjoint operator obtained by closing the positive quadratic form $\int_0^\infty(|f'|^2+(u_i+u_j)|f|^2)\,dx$, initially on smooth compactly supported functions. Thus $K_{ij}$ is the integral kernel of $H_{ij}^{-1}$. In particular, its derivative in the first variable has jump $-1$ at the second variable. The potential is $12/x^2+O(x^2)$ at zero and tends to the positive constant $\sqrt{\eta_i}+\sqrt{\eta_j}$ at infinity. Its two independent local solutions have orders $x^4$ and $x^{-3}$. The regular solution at zero and the decaying solution at infinity can therefore be normalized by
\[
 \phi_{ij}(x)\sim x^4/7,\qquad \psi_{ij}(x)\sim x^{-3}\quad(x\downarrow0),
 \qquad \psi_{ij}(\infty)=0.
\]
With these normalizations $\phi_{ij}'\psi_{ij}-\phi_{ij}\psi_{ij}'=1$ and
$K_{ij}(x,y)=\phi_{ij}(\min(x,y))\psi_{ij}(\max(x,y))$.

For a finite label set $S$ with $|S|=k\ge3$, a \emph{cubic scheme with faces labelled by $S$} is such a planar map with every vertex of degree three and with its faces bijectively labelled by $S$. It has $2k-4$ vertices and $3k-6$ edges.

Fix such a scheme $G$, and write $V(G)$ and $E(G)$ for its vertex and edge sets. Define the finite set
\[
 \mathcal A_G=\{\alpha:S\longrightarrow E(G):
 \alpha(i)\text{ is incident to face }i\text{ for every }i\in S\}.
\]
Thus each face chooses an incident edge. An edge incident to the same face on both sides occurs only once among that face's possible choices. The map $\alpha$ need not be injective: two distinct faces can choose their common edge.

Let $x=(x_v)_{v\in V(G)}\in(0,\infty)^{V(G)}$ and $z=(z_i)_{i\in S}\in(0,\infty)^S$. For an edge $e$, write $v,w$ for its endpoint incidences and $i,j$ for its incident face labels, allowing $v=w$ and $i=j$. Put $S_e(\alpha)=\alpha^{-1}(\{e\})\subseteq\{i,j\}$. We define $F_\alpha(e;x,z)$ by the following four cases. If $S_e(\alpha)=\varnothing$, put
\[
 F_\alpha(e;x,z)=K_{ij}(x_v,x_w).
\]
If $i\ne j$ and $S_e(\alpha)=\{h\}$, where $h\in\{i,j\}$, put
\[
 F_\alpha(e;x,z)=K_{ij}(x_v,z_h)p_h(z_h)K_{ij}(z_h,x_w).
\]
If $i=j$ and $\alpha(i)=e$, put
\[
 F_\alpha(e;x,z)=2K_{ii}(x_v,z_i)p_i(z_i)K_{ii}(z_i,x_w).
\]
If $i\ne j$ and $\alpha(i)=\alpha(j)=e$, put
\begin{align*}
 F_\alpha(e;x,z)={}&K_{ij}(x_v,z_i)p_i(z_i)K_{ij}(z_i,z_j)p_j(z_j)K_{ij}(z_j,x_w)\\
 &+K_{ij}(x_v,z_j)p_j(z_j)K_{ij}(z_j,z_i)p_i(z_i)K_{ij}(z_i,x_w).
\end{align*}
An arbitrary ordering of the two endpoint incidences of $e$ may be used: symmetry of $K_{ij}$ preserves the first three cases and interchanges the two summands in the fourth. Both summands remain for a loop, even if their values coincide. Distinct parallel edges remain distinct elements of $E(G)$.

Let $\mathcal G_S$ be the set of isomorphism classes of cubic schemes with faces labelled by $S$, and let $\Aut(G)$ be the group of orientation-preserving, face-label-preserving automorphisms. We define
\begin{equation}\label{eq:scheme-sum-definition}
 \mathcal I_S(\eta)=\sum_{[G]\in\mathcal G_S}\frac1{|\Aut(G)|}
 \sum_{\alpha\in\mathcal A_G}
 \int_{(0,\infty)^{V(G)}}\int_{(0,\infty)^S}
 \prod_{e\in E(G)}F_\alpha(e;x,z)\,
 \prod_{i\in S}dz_i\prod_{v\in V(G)}dx_v.
\end{equation}
We allow $+\infty$ until finiteness is proved.

\begin{Theorem}[Critical limits of generating functions for labelled planar maps]\label{ms:prop:general-analytic}
For every fixed $k\ge3$, the integral $\mathcal I_{\{1,\ldots,k\}}$ is finite and continuous on $(0,\infty)^k$, and
\[
 \varepsilon^{4k-10}\mathcal Z_k(g_1(\varepsilon),\ldots,g_k(\varepsilon))
 \longrightarrow 3^{3k-6}\mathcal I_{\{1,\ldots,k\}}(\lambda/4)
\]
locally uniformly on positive rate sets.

\end{Theorem}
In words, the generating series multiplied by $\varepsilon^{4k-10}$ tends to $3^{3k-6}$ times the marked scheme integral evaluated at $\lambda/4$; the resulting integral is finite and continuous throughout the positive parameter domain. Here ``locally uniformly'' means uniformly for $\lambda$ in every compact subset of $(0,\infty)^k$. The integral uses all cubic schemes, including loops, parallel edges and repeated incidences of a face. It is an explicit finite-dimensional representation; no evaluation of the integrals in elementary functions is claimed.

\begin{Remark}
The result does not address real analyticity, uniform convergence as any $\lambda_i$ approaches zero, or the regime where the number of faces grows.
\end{Remark}

The normalization can already be read from a separated marked cubic scheme. Inserting one distinct degree-two vertex for each of the $k$ face marks gives $3k-4$ label variables and $4k-6$ chain segments. A label sum contributes $\varepsilon^{-1}$, a chain kernel contributes $\varepsilon^{-1}$, and a mark contributes $\varepsilon^3$. The combined order is consequently
\[
 \varepsilon^{-(3k-4)}\varepsilon^{-(4k-6)}\varepsilon^{3k}
 =\varepsilon^{-(4k-10)}.
\]
The proof must justify this count uniformly and exclude all diagrams for which it is not the correct leading order. Sections~\ref{sec:reduction} and~\ref{sec:kernels} establish the marking expansion and its local weights. Section~\ref{sec:subsets} gives the planar subset estimate. Section~\ref{sec:limit} uses it first to bound the errors and then to pass to the integral.

\section{Trees, chains, and a positive marking expansion}\label{sec:reduction}
\subsection{The reduced scheme and its local weights}\label{sec:trees}
To compare the discrete series with the marked scheme integral, we first sum the attached trees and chains and then expand the minimum-label constraint into corner marks. The same weights will describe both the main diagrams and the error diagrams.

We use the pruning and chain-suppression decomposition of Miermont \cite{miermont}*{Section~3.1}; Bouttier and Guitter \cite{bouttierguitter}*{Section~3.2} describe its three-face version. The combinatorial operations apply without any change to our face minima, which will be imposed after the decomposition. Prune edges incident to vertices of degree one, then suppress every maximal degree-two chain. The resulting reduced scheme has minimum degree three and retains the ribbon structure and face numbers. The pruned plane trees are attached in corners, the sectors between consecutive incidences at a vertex. Their root labels are the labels of their attachment vertices. A chain is nonempty, even when its endpoint labels coincide; a loop retains both endpoint incidences. Euler's relation bounds the numbers of vertices and edges of a reduced scheme by $2k-4$ and $3k-6$, respectively, so there are finitely many schemes for fixed $k$.

Let $R_l(g)$ count plane trees planted in a corner with root label $l$, all labels at least one, adjacent increments in $\{-1,0,1\}$, and weight $g$ per edge. The one-vertex tree has weight one; put $R_l=0$ for $l\le0$. The explicit tree series is \cite{bouttierguitter}*{Section~4.1, equation~(4.2), and equations~(1.4)--(1.5)}
\begin{ExternalTheorem}\label{ext:labelled-tree-series}
For $0<g<1/12$, define
\[
 R_\infty(g)=\frac{1-\sqrt{1-12g}}{6g},\qquad
 \frac{q}{1+q+q^2}=gR_\infty(g)^2,\quad 0<q<1.
\]
Then
\[
 R_l(g)=R_\infty(g)\frac{(1-q^l)(1-q^{l+3})}{(1-q^{l+1})(1-q^{l+2})},\qquad l\ge1.
\]
\end{ExternalTheorem}
For faces labelled $i,j$, define a one-step chain matrix on the positive integers by
\[
 T_{ij}(l,m)=\sqrt{g_i g_j}
 \sqrt{R_l(g_i)R_l(g_j)R_m(g_i)R_m(g_j)}\,
 \ind_{\{|l-m|\le1\}},\qquad G_{ij}=\sum_{r\ge1}T_{ij}^r.
\]
Here $T_{ij}^r(l,m)$ sums over all intermediate labels on a chain of $r$ edges from label $l$ to label $m$. The sum starts at $r=1$ because a scheme edge is replaced by a nonempty chain. These series converge when every $g_i<1/12$. Indeed their row sums are at most
$3\sqrt{g_i g_j}R_\infty(g_i)R_\infty(g_j)<1$, since $\sqrt gR_\infty(g)<1/\sqrt3$ for $g<1/12$. The same estimate holds with a finite upper label cutoff.

The endpoint square roots in a chain supply $\sqrt{R_l(g_i)}$ on each incident side of face $i$. At a retained vertex, the two sides bordering each corner combine to exactly one factor $R_l(g_i)$. This remains true at loops and at any vertex degree. Hence summing the internal labels and nonzero lengths of every segment gives exactly one $G_{ij}$ per segment, with no further vertex factor. A corner selected on a face is remembered separately if another corner of the same face occurs at the same vertex.

\subsection{Face minima as corner marks}
To isolate the leading scheme contribution, we mark one corner on every face. Contacts and one repeated face mark will bound the error in replacing a face-minimum condition by the sum of its possible marks.

Vertices of a reduced scheme are called \emph{original vertices}. In a chain replacing a scheme edge, all other vertices have degree two. A mark is a chosen corner, with the weight below. We use one mark for each face in the leading term; the error estimate also uses one additional mark on one face.

For a corner of face $i$ at label $l$, put
\[
 \delta_i(l)=1-\frac{R_{l-1}(g_i)}{R_l(g_i)},\qquad
 \widehat B_i=\sum_{c\text{ of face }i}\delta_i(l_c),\qquad
 B_i=1-\prod_{c\text{ of face }i}(1-\delta_i(l_c)).
\]
The exact marking reduction below will establish how $B_i$ enforces minimum label one on face $i$ and will justify the associated error bound. Here we define the two classes of marked expressions needed for that argument. Expanding $\prod_i\widehat B_i$ selects one corner of every face, which we call the main case. We also allow one additional chosen corner on one face, which we call the repeated-face case. Each chosen corner carries its factor $\delta_i$. Suppress unmarked degree-two chain vertices and retain every marked chain vertex as an inserted degree-two vertex. Marks are called separated when they lie at distinct internal chain vertices. All remaining configurations are called discrete contacts: either a mark lies at an original vertex or several marks share one chain vertex. This definition concerns graph vertices, not numerical equality of labels at different vertices.

For a vertex set $S$ of such a marked graph, let $O_S,I_S,M_S$ count its original vertices, inserted vertices, and marks. Put $C_S=M_S-I_S$. Let $E(S)$ count internal edges and let $c_S=|\partial S|$ count edges from $S$ to its complement, including multiplicities. Let $f(S)$ count distinct global faces incident to $S$. In the repeated-face case put $r_S=1$ if both marks of that face lie in $S$, and put $r_S=0$ otherwise; in the main case always put $r_S=0$. Define the two exponents
\[
 L(S)=|S|+|E(S)|+3c_S-3M_S,\qquad
 D(S)=|S|+|E(S)|-2c_S-3M_S.
\]
Order the marked-graph vertices by nondecreasing labels, breaking ties arbitrarily. A low-label set is an initial segment of this ordering; it need not be connected. The low-label exponent of such a set $S$ means $L(S)$. The exponent $L(S)$ measures the power gained when the coordinates in $S$ approach zero together. The exponent $D(S)$ will measure growth when labels in $S$ become large. We define them for arbitrary vertex sets because an initial or final segment of a label ordering need not be connected.

For the full vertex set, write $v=O_{\mathrm{all}}$, $I=I_{\mathrm{all}}$, and $M=M_{\mathrm{all}}$ for the total numbers of original vertices, inserted vertices, and marks, respectively. Put $B=2v-2$ and $C=C_{\mathrm{all}}=M-I$, the total contact count. Write $r=M-k\in\{0,1\}$ for the total surplus of marks.
The quantities $C_S$ and $C$ are nonnegative. Every inserted vertex carries at least one mark; marks at original vertices and extra marks at inserted vertices account exactly for $C=M-I$. In particular, $C=0$ means that every mark is at its own internal chain vertex. We call a graph obtained after retaining the marked vertices a \emph{marked diagram}; its edges are the remaining chain segments.
\begin{Proposition}[Positive marking reduction]\label{ms:lem:v8-marking-reduction}
For each fixed $k\ge3$, the series $\mathcal Z_k$ is the sum over finitely many reduced schemes of their decorated skeleton weights multiplied by $\prod_i B_i$. The contribution obtained by replacing this factor by $\prod_i\widehat B_i$ counts one chosen corner of each face. The replacement error is nonnegative and bounded by a finite sum of marked diagrams with one additional mark on one face. After suppressing unmarked chain vertices, every term is a product of one $G_{ij}$ per remaining segment and one $\delta_i$ per mark, with the original inverse-automorphism weights and corner multiplicities. All identities first hold with a finite skeleton-label cutoff and then coefficientwise without that cutoff.
\end{Proposition}
\begin{proof}
Prune the trees and suppress degree-two chains as described above. For a reduced $k$-face scheme, Euler's formula gives $e=v+k-2$ and the minimum-degree condition gives $v\le2k-4$, so only finitely many schemes occur. Equality holds exactly for cubic schemes. Inverse automorphism weights are preserved: for a fixed scheme, its finite automorphism group acts on decorations, and the sum of inverse stabilizer orders over decoration orbits is the sum over decorations divided by the group order. This also applies to marks and their orders.

Initially impose a finite cutoff on all skeleton labels. Fix a labelled skeleton. Its edge and corner-tree weight is
\[
 W=\prod_e\sqrt{g_{i(e)}g_{j(e)}}\prod_c R_{l_c}(g_{i(c)}).
\]
Raising the lower label bound of face $i$ by one multiplies its corner factors by $1-\delta_i(l_c)$, including the exclusion of an incident vertex of label one because $R_0=0$. Inclusion--exclusion over the $k$ faces therefore gives the exact weight $W\prod_iB_i$. Since $0\le\delta_i\le1$,
\[
 0\le B_i\le \widehat B_i,\qquad
 0\le \widehat B_i-B_i\le\sum_{c<d}\delta_i(l_c)\delta_i(l_d).
\]
Telescoping the difference of the two products bounds $\prod_i\widehat B_i-\prod_iB_i$ by a sum with one extra mark of one face label. Thus it suffices to evaluate the main marked diagrams and bound the repeated-face diagrams. At fixed size $n$, every retained label is at most $n+1$, so increasing the cutoff recovers the exact series coefficientwise. For positive weights, all skeleton contributions and their pointwise replacement errors are nonnegative. Their cutoff sums therefore increase to the corresponding untruncated sums; monotone convergence preserves the positive expansions and error bounds, allowing $+\infty$ at this stage. Suppressing unmarked chain vertices expresses each as a product of finitely many chain kernels and mark factors, summed over its remaining labels. The number of resulting graphs depends only on $k$.
\end{proof}

\section{Uniform estimates for trees and chains}\label{sec:kernels}
All estimates in this section are uniform when $\lambda$ ranges over a fixed compact subset of $(0,\infty)^k$ and $\varepsilon$ is sufficiently small. In particular, all discrete weights satisfy $0<g_i(\varepsilon)<1/12$. The continuum parameters are always $\eta_i=\lambda_i/4$.

\subsection{The corner insertion}
A corner mark is a discrete difference of the tree partition function. Its scaling limit is therefore the derivative of the tree profile, but a bound down to the first admissible label will also be needed.
\begin{Lemma}[Corner weights]\label{ms:lem:v8-corner-limit}
The corner weight decays at least cubically in its label, with an additional exponential cutoff on the scale $\varepsilon^{-1}$. After rescaling by $\varepsilon^{-3}$, it tends to one third of the continuum insertion $p_i$. More precisely, uniformly on compact positive rate sets, $0\le\delta_i(l)\le Cl^{-3}e^{-c\varepsilon l}$ for $l\ge1$, and locally uniformly for $x>0$,
\[
 \varepsilon^{-3}\delta_i(\lfloor x/\varepsilon\rfloor)\longrightarrow p_i(x)/3.
\]
\end{Lemma}
\begin{proof}
Use Theorem~\ref{ext:labelled-tree-series}, writing $q_i$ for the parameter $q$ associated with $g_i$. Direct substitution in $1-R_{l-1}/R_l$ gives the exact numerator factorization
\[
 \delta_i(l)=\frac{q_i^{l-1}(1-q_i)^3(1+q_i)(1+q_i^{l+1})}
 {(1-q_i^l)^2(1-q_i^{l+3})}.
\]
The defining equation for $q_i$ gives $-\log q_i=\lambda_i^{1/4}\varepsilon+O(\varepsilon^3)$ uniformly on compact positive rate sets. Using $1-e^{-s}\asymp\min(s,1)$ in this factorization proves the stated bound for $\delta_i$. Expanding the same expression for $l=\lfloor x/\varepsilon\rfloor$ proves its locally uniform limit $-u_i'(x)/3=p_i(x)/3$.

The profile itself is the established continuum two-point function recalled by Bouttier and Guitter \cite{bouttierguitter}*{equations~(1.1) and~(1.7)--(1.9)}. In their notation, set $\epsilon_{\mathrm{BG}}=\varepsilon^4$, $\Lambda=\lambda_i/36$ and $D=x$. Their parameter is then $\alpha=\sqrt{3/2}\,\Lambda^{1/4}=\lambda_i^{1/4}/2$, and their functions satisfy
\[
 \mathcal F(x;\alpha)=-u_i(x)/3,\qquad
 \mathcal G(x;\alpha)=\partial_x\mathcal F(x;\alpha)=p_i(x)/3.
\]
For $l\ge2$, their discrete two-point function is $G(l;g_i)=\log(R_l(g_i)/R_{l-1}(g_i))$, so $\delta_i(l)=1-\exp(-G(l;g_i))$ has the same leading profile. The exact factorization above supplies the additional uniformity on compact positive rate sets and the exponential estimate valid for every $l\ge1$, which are needed here beyond identification of the profile.
\end{proof}

\subsection{The chain resolvent}
A chain can be summed as the resolvent of a nearest-neighbour transfer matrix. We first compare this resolvent with explicit critical and massive kernels. This provides both its boundary decay and its decay across a long label difference.
\begin{Lemma}[Chain kernels]\label{ms:lem:v8-chain-limit}
The chain kernel has a cubic vanishing bound in its smaller label, an inverse-square bound in its larger label, and exponential decay in their difference. Rescaling the kernel by $\varepsilon$ gives three times the continuum Green kernel. More precisely, uniformly on compact positive rate sets, the unshifted infinite-cutoff chain satisfies
\[
0\le G_{ij}(l,m)\le C\min(l,m)^3\max(l,m)^{-2}e^{-c\varepsilon|l-m|}.
\]
Moreover $\varepsilon G_{ij}(\lfloor x/\varepsilon\rfloor,\lfloor y/\varepsilon\rfloor)\to3K_{ij}(x,y)$ locally uniformly for positive $x,y$ and rates. Both statements include coincident face labels and coincident rates.
\end{Lemma}
\begin{proof}
Use the explicit tree formula in Theorem~\ref{ext:labelled-tree-series}. We first specify its critical endpoint. As $g\uparrow1/12$, we have $R_\infty(g)\uparrow2$ and $q\uparrow1$. Hence, for every integer $l\ge1$,
\[
 R_l^{\mathrm{crit}}:=\lim_{g\uparrow1/12}R_l(g)
 =2\frac{l(l+3)}{(l+1)(l+2)}.
\]
Indeed $(1-q^j)/(1-q)\to j$ for every positive integer $j$. Since the tree coefficients are nonnegative, monotone convergence identifies this finite limit with the tree sum at $g=1/12$. Put
\[
 w_l^{\mathrm{crit}}=\frac1{12}(R_l^{\mathrm{crit}})^2,
 \qquad
 V_l^{\mathrm{crit}}=(w_l^{\mathrm{crit}})^{-1}-3
 =3\left(\frac{(l+1)^2(l+2)^2}{l^2(l+3)^2}-1\right).
\]
Critical inverse-kernel entries below mean the increasing limits of the Green-kernel entries on $\{1,\ldots,N\}$ with zero boundary values at $0$ and $N+1$, as $N\to\infty$.

For the subcritical kernel, write
\[
 w_l=\sqrt{g_ig_j}R_l(g_i)R_l(g_j),\qquad
 \mathcal H=\operatorname{diag}(w_l^{-1})-A,
\]
where $A(l,m)=\ind_{\{|l-m|\le1\}}$ on the positive integers with a zero boundary value at $0$. The symmetric transfer matrix is $T=\operatorname{diag}(\sqrt w)A\operatorname{diag}(\sqrt w)$, so
\[
 G=\operatorname{diag}(w^{-1/2})\mathcal H^{-1}
       \operatorname{diag}(w^{-1/2})-I.
\]
The factors $w_l$ are bounded above and bounded away from zero uniformly in $l$ for sufficiently small $\varepsilon$ on the given rate set. Relative to the Dirichlet discrete Laplacian $3I-A$, the potential $w_l^{-1}-3$ dominates the critical potential and a positive constant times $\varepsilon^2$. The first domination follows from coefficient monotonicity of $R_l$ in $g$; the second follows from $R_l\le R_\infty$ and the explicit expression for $R_\infty$.

At criticality, direct substitution gives a decaying solution and its regular companion:
\[
 y_l=\frac{4(2l+3)}{(l+1)^2(l+2)^2},\qquad
 \phi_l=y_l\sum_{r=0}^{l-1}(y_ry_{r+1})^{-1}.
\]
Their discrete Wronskian has magnitude one. Thus the critical inverse kernel is \[\phi_{\min(l,m)}y_{\max(l,m)},\] bounded by
$C\min(l,m)^4\max(l,m)^{-3}$. The inverse of the free Dirichlet discrete Laplacian plus $c\varepsilon^2$ is bounded by $C\min(l,m)e^{-c'\varepsilon|l-m|}$, as follows by its two geometric homogeneous solutions. Write $\mathcal L=3I-A$ and $V_l=w_l^{-1}-3$, retaining the critical potential defined above. Choose $c_0>0$ such that $V_l\ge c_0\varepsilon^2$ uniformly. The preceding comparisons give
\[
 V_l\ge\frac23V_l^{\mathrm{crit}}+\frac13c_0\varepsilon^2.
\]
Set $\mathcal H_{\mathrm{crit}}=\mathcal L+\operatorname{diag}(V^{\mathrm{crit}})$ and $\mathcal H_{\mathrm{free},\varepsilon}=\mathcal L+c_0\varepsilon^2I$. The Feynman--Kac formula represents their inverse kernels by integrals over time and over paths of the same continuous-time nearest-neighbor walk, killed on reaching zero. Applying H\"older's inequality with exponents $3/2$ and $3$ to this common path-and-time measure gives
\[
 \mathcal H^{-1}(l,m)
 \le\bigl(\mathcal H_{\mathrm{crit}}^{-1}(l,m)\bigr)^{2/3}
     \bigl(\mathcal H_{\mathrm{free},\varepsilon}^{-1}(l,m)\bigr)^{1/3}.
\]
The two inverse kernels are the critical and free kernels estimated above. The displayed relation between $G$ and $\mathcal H^{-1}$ then gives
\[
 G(l,m)\le C\min(l,m)^3\max(l,m)^{-2}
 e^{-c\varepsilon|l-m|}.
\]
These comparisons hold with a finite cutoff and pass monotonically to the infinite half-line.

To identify the limit, the tree expression also gives, locally uniformly for $x>0$,
\[
 R_{\lfloor x/\varepsilon\rfloor}(g_i)
 =2\bigl(1-\varepsilon^2u_i(x)/3+o(\varepsilon^2)\bigr),
 \qquad
 w_{\lfloor x/\varepsilon\rfloor}^{-1}
 =3+\varepsilon^2(u_i(x)+u_j(x))+o(\varepsilon^2).
\]
Consequently the homogeneous discrete recurrence, divided by $\varepsilon^2$, converges on compact intervals to $H_{ij}f=0$. We spell out the compactness argument with both endpoints allowed to vary. Put
\[
 h_\varepsilon(l,m)=\varepsilon\mathcal H^{-1}(l,m),
 \qquad V_l=w_l^{-1}-3.
\]
For fixed $m$, the source equation is
\[
 2h_\varepsilon(l,m)-h_\varepsilon(l-1,m)-h_\varepsilon(l+1,m)
       +V_lh_\varepsilon(l,m)=\varepsilon\ind_{\{l=m\}}.
\]
Thus the discrete slope changes by $V_lh_\varepsilon(l,m)/\varepsilon$ at an ordinary site and has an additional drop of one at the source. The kernel envelope bounds $h_\varepsilon$ on every compact set of positive rescaled endpoints, and $V_l=O(\varepsilon^2)$ there. Summing the slope changes over an interval and comparing with the bounded endpoint values gives a uniform bound on the slopes, including the one source jump. The piecewise linear interpolations are therefore locally equicontinuous in the first endpoint, uniformly for the second endpoint in a compact interval. Symmetry of $\mathcal H^{-1}$ gives the same control in the second endpoint. The piecewise bilinear interpolations are consequently jointly locally equicontinuous, including across the diagonal.

Every joint subsequential limit solves the homogeneous differential equation away from the limiting source $y$. Summing the displayed source equation over a small interval containing $y$ shows that the slope drops by one, up to the potential term integrated over that interval. Letting its length tend to zero gives derivative jump $-1$. Equicontinuity permits the source labels themselves to vary with $\varepsilon$, provided their rescaled values tend to $y$. The envelope excludes the $x^{-3}$ solution at zero and imposes decay at infinity. Uniqueness of this Green boundary problem therefore identifies the joint limit as $K_{ij}$ for $\varepsilon\mathcal H^{-1}$. The diagonal subtraction $\varepsilon I$ vanishes and $w_l^{-1/2}w_m^{-1/2}\to3$, giving the asserted $3K_{ij}$ limit for $\varepsilon G$. All these bounds hold along arbitrary convergent sequences of rates in a compact positive set. The same compactness and uniqueness argument therefore proves local uniformity in the rates and both endpoints together.
\end{proof}
\begin{Remark}\label{ms:lem:three-face-kernels}
We will use the preceding two lemmas together. They say that a segment and a mark have respective rescaled limits $3K_{ij}(x,y)$ and $p_i(x)/3$, and give the common bounds
\[
 G_{ij}(l,m)\le C\frac{\min(l,m)^3}{\max(l,m)^2}e^{-c\varepsilon|l-m|},
 \qquad \delta_i(l)\le Cl^{-3}e^{-c\varepsilon l}.
\]
No distinctness assumption is made on the face labels, the rates, or the endpoint labels. Thus diagonal kernels and equal-parameter kernels require no separate limiting argument.
\end{Remark}

\section{Planarity and vertex-subset exponents}\label{sec:subsets}
The singularity of a diagram is governed by sets of vertices whose labels have a common scale. A set may be disconnected and may meet a face several times. We first bound the number of distinct faces that a connected set can see, using the ribbon structure rather than simple facial cycles.
\begin{Lemma}[Incident faces]\label{ms:lem:ribbon-subset-face-bound}
Let $G$ be a finite connected ribbon graph embedded in the sphere, allowing loops and multiple edges. Let $B$ be a nonempty proper vertex subset whose induced subgraph is connected. Write $e_B$ for the number of its internal edges, $c_B$ for the number of edges with exactly one endpoint in $B$, and $f_B$ for the number of distinct faces of the full graph incident to a vertex of $B$. Count edges with multiplicity and incidences through ribbon corners. Then $c_B\ge1$ and
\[
f_B\le e_B-|B|+1+c_B.
\]
In particular, suppose every vertex of $B$ has degree two or three in the full graph. If $O$ and $D$ are respectively the numbers of its degree-three and degree-two vertices, then
\[
e_B=\frac{3O+2D-c_B}{2},\qquad
f_B\le\frac O2+\frac{c_B}{2}+1.
\]
\end{Lemma}
\begin{proof}
Connectivity of $G$ and properness of $B$ give at least one edge leaving $B$. Take a small closed regular ribbon neighborhood of the connected induced subgraph. It has genus zero, retracts onto that subgraph, and therefore has
\[
b=2-(|B|-e_B)=e_B-|B|+2
\]
boundary circles. This includes the case of a single vertex with no internal edge. Each exiting halfedge meets this boundary once, giving $c_B$ marked points, distributed over $r\ge1$ of its circles. A circle containing $t>0$ such points has $t$ intervening arcs; a circle with no such points contributes one whole circle. Each arc or whole circle follows one face of the full ribbon graph. Every face incident to $B$ is represented in this way, since a corner at a vertex of $B$ lies on one of these boundary pieces. Several pieces may represent the same global face, which only decreases the count. Thus
\[
f_B\le c_B+(b-r)\le c_B+b-1=e_B-|B|+1+c_B.
\]
The argument counts boundary pieces rather than simple facial cycles, so it includes loops, bridges and repeated visits to a face. Finally, the degree sum in $B$ is $2e_B+c_B=3O+2D$. Substituting $|B|=O+D$ gives the last formulas.
\end{proof}
Recall from Section~\ref{sec:reduction} that $L(S)$ is the exponent for lowering the labels in $S$, while $D(S)$ is the exponent for raising them. The next lemma gives the lower bounds needed for the continuum integral and the upper bounds needed for discrete error estimates. In its statement $B=2v-2$, where $v$ is the number of original scheme vertices, and $r=M-k\in\{0,1\}$.
\begin{Lemma}[Subset exponents]\label{ms:lem:general-marked-bounds}
A connected proper vertex set has a low-label exponent bounded below by its original vertices and its exiting edges, with explicit losses for repeated marks and contacts. Every proper vertex set, including a disconnected one, has a high-label exponent at least three below $B$. Precisely, for a connected proper nonempty vertex set $S$,
\[
 f(S)\le |E(S)|-|S|+1+c_S,\qquad
 L(S)\ge 2O_S+2c_S-1-r_S-2C_S\ge3-r_S-2C_S.
\]
For every proper nonempty vertex set $H$, connected or not,
\[
 D(H)\le B-3,\qquad D(\mathrm{all})=B-r-2C.
\]
For a cubic scheme with separated main marks, $L(\mathrm{all})=4k-10>0$, and $L(S)>0$ for every proper nonempty low-label set $S$, meaning an initial segment of the marked-graph vertices in nondecreasing label order.
\end{Lemma}
\begin{proof}
The marked graph is a connected planar ribbon graph, and $S$ is a nonempty proper vertex set with connected induced subgraph. Lemma~\ref{ms:lem:ribbon-subset-face-bound} therefore gives the stated face bound, including loops and multiple edges.

Let $d_w\ge3$ be the degrees of the original vertices. The degree sum gives
\[
 |E(S)|=\frac12\left(\sum_{w\in S\text{ original}}d_w+2I_S-c_S\right),
 \qquad
 f(S)\le\sum_{w\in S\text{ original}}(d_w/2-1)+c_S/2+1.
\]
Each face label occurs at most once except for the specified repeated face label, so $M_S\le f(S)+r_S$. Inserting $I_S=M_S-C_S$ into the definition of $L(S)$ and then using this bound proves
\[
 L(S)\ge2O_S+2c_S-1-r_S-2C_S.
\]
If $O_S\ge1$, then $c_S\ge1$. If $O_S=0$, connectedness makes $S$ an interval of inserted vertices on a scheme edge, and $c_S\ge2$; it cannot be a closed cycle, since each subdivided scheme edge has original endpoints. This proves the second lower bound. For disconnected $S$, both the exponent and the mark counts add over its connected components.

Euler's formula gives $e=v+k-2$ edges in the reduced scheme. After insertion the graph has $v+I$ vertices and $e+I$ edges. Thus
\[
 D(\mathrm{all})=v+e+2I-3M=2v-2-r-2C.
\]
For $S$ complementary to $H$, direct subtraction of edge counts gives
\[
 D(H)=D(\mathrm{all})-L(S).
\]
If $S$ has $b\ge1$ connected components, their total repeated-face surplus is at most $r$, and their total contact count is $C_S$. Their lower bounds yield
\[
 D(H)\le B-r-2C-3b+\sum r_{S_j}+2C_S
 \le B-3b-2C_H\le B-3.
\]
Finally, a cubic $k$-face scheme has $v=2k-4$. In the separated main case $r=C=0$, so the full exponent is $B=4k-10$. For any proper nonempty vertex set $S$, each of its connected components is a proper set with $r_{S_j}=C_{S_j}=0$, so the proved bound gives $L(S_j)\ge3$. Additivity yields $L(S)=\sum_jL(S_j)\ge3\,\#\{\text{connected components of }S\}>0$. Every proper nonempty low-label set is one such vertex set, irrespective of connectivity or the chosen label ordering. This proves the last assertion.
\end{proof}

\section{Discarding errors and taking the critical limit}\label{sec:limit}
\subsection{Discrete estimates with the lower cutoff retained}
The extra-mark expansion and the main marking expansion produce only finitely many types of diagrams for fixed $k$. The following estimate treats all nonsurviving types at once. Its use of the original cutoff $l\ge1$ matters: a contact diagram need not define an integrable continuum expression at zero.
\begin{Proposition}[Error diagrams]\label{ms:lem:v8-diagram-errors}
Fix a reduced scheme with $k\ge3$ numbered faces and $v$ original vertices. Choose its marked corners and their order along every scheme edge, and suppress all unmarked degree-two chain vertices. Write $\mathcal G$ for the resulting graph, $V$ for its retained vertices (the original vertices and the marked internal chain vertices), and $\mathcal E$ for its segment multiset. Each segment $e=uw$ has incident face labels $i(e),j(e)$, possibly equal. Let $\mathcal M$ index the individual marks; for $a\in\mathcal M$, let $h(a)$ be its face label and $m(a)\in V$ its vertex. Several marks may have the same vertex. At the weights $g_i(\varepsilon)$ from \eqref{eq:critical-scaling}, define the nonnegative sum
\[
 T_{\mathcal G}(\varepsilon;\lambda)
 =\sum_{l\in\{1,2,\ldots\}^{V}}
   \prod_{e=uw\in\mathcal E}G_{i(e)j(e)}(l_u,l_w)
   \prod_{a\in\mathcal M}\delta_{h(a)}(l_{m(a)}).
\]
There is no distinct-label restriction. A loop supplies its diagonal kernel factor; parallel segments and individual marks are counted with multiplicity. This is the sum for a fixed choice of marked corners and their orders, before multiplication by the fixed inverse-automorphism weight.

In a main diagram every face has exactly one mark. It is a main contact diagram if a mark lies at an original vertex or two marks share an inserted vertex. A repeated-face diagram has two marks on one face and one on every other face, with contacts allowed. A noncubic main diagram is a main diagram whose original reduced scheme is not cubic, equivalently $v<2k-4$; it may have contacts or separated marks. Put $B=2v-2$ and abbreviate $T_{\mathcal G}(\varepsilon;\lambda)$ to $T_{\mathcal G}(\varepsilon)$. Uniformly on compact positive rate sets, with $b=3k-3$, this sum has the following bounds:
\[
\begin{array}{ll}
\text{main contacts:}&O\!\left(\varepsilon^{-\max(B-2,0)}(1+|\log\varepsilon|)^b\right),\\
\text{repeated-face diagrams:}&O\!\left(\varepsilon^{-\max(B-1,0)}(1+|\log\varepsilon|)^b\right),\\
\text{noncubic main diagrams:}&O\!\left(\varepsilon^{-\max(B,0)}(1+|\log\varepsilon|)^b\right).
\end{array}
\]
Each sum is finite at every sufficiently small positive $\varepsilon$. Multiplication by $\varepsilon^{4k-10}$ sends all three classes to zero locally uniformly, including schemes with one original vertex.
Multiplying by a fixed diagram weight preserves these bounds, and summing over the finitely many corner choices and mark orders only changes the implicit constants.
\end{Proposition}
\begin{proof}
Use the kernel and corner bounds of Remark~\ref{ms:lem:three-face-kernels} and the vertex-set exponents of Lemma~\ref{ms:lem:general-marked-bounds}. We first bound the discrete error diagrams with the lower label cutoff fixed at $1$. This is necessary for contacts: their low-set exponents need not be positive. Fix one marked diagram $\mathcal G$ with $N$ retained vertices, write $\nu(v)$ for its mark multiplicity, and let $T_{\mathcal G}(\varepsilon)$ be its nonnegative label sum. The kernel and mark envelopes dominate it by a constant times
\[
 \sum_{l_v\ge1}
 \prod_{uv\in E(\mathcal G)}
       \frac{\min(l_u,l_v)^3}{\max(l_u,l_v)^2}
 \prod_v l_v^{-3\nu(v)}
 \exp\!\left\{-c\varepsilon
       \left(\sum_{uv\in E(\mathcal G)}|l_u-l_v|
                     +\sum_v\nu(v)l_v\right)\right\}.
\]
Loops and multiple edges are included with multiplicity. There is at least one mark. By connectedness, a path from a vertex of maximum label to any marked vertex gives
\[
 \max_v l_v\le
 \sum_{uv\in E(\mathcal G)}|l_u-l_v|+\sum_v\nu(v)l_v.
\]
Thus the displayed exponential is at most $e^{-c\varepsilon\max l_v}$.

Partition each label into blocks $2^{n_v}\le l_v<2^{n_v+1}$ with $n_v\in\mathbb Z_{\ge0}$. A block contains at most $2^{n_v}$ integer labels. Up to a constant depending only on the diagram, the contribution of a fixed collection of blocks is therefore bounded by
\[
 2^{P(n)}e^{-c\varepsilon2^{\max n_v}},\qquad
 P(n)=\sum_v(1-3\nu(v))n_v
       +\sum_{uv\in E(\mathcal G)}
                    \bigl(3\min(n_u,n_v)-2\max(n_u,n_v)\bigr).
\]
Order the vertices so that $n_1\le\cdots\le n_N$, with any choice for ties; summing over at most $N!$ such orders only changes the constant. Put $d_1=n_1$, $d_h=n_h-n_{h-1}$ for $h\ge2$, and $H_h=\{v_h,\ldots,v_N\}$. Every $d_h$ is nonnegative and $\sum_hd_h=n_N$. Increasing $d_1$ raises all scales, whereas increasing $d_h$ for $h\ge2$ raises exactly the final set $H_h$. The measure, internal edges, crossing edges, and marks then give the exact identity
\[
 P(n)=D(\mathrm{all})d_1+\sum_{h=2}^N D(H_h)d_h.
\]
For example, a crossing edge has its upper endpoint in $H_h$ and contributes $-2$ to its coefficient; an internal edge, including a loop, contributes $1$. This verifies the identity also for disconnected final sets.

Define
\[
 d_* =\max\bigl(\{0,D(\mathrm{all})\}
       \cup\{D(H):\varnothing\ne H\subsetneq V(\mathcal G)\}\bigr).
\]
Then $P(n)\le d_*n_N$. For $n_N=R$ the number of nonnegative gap vectors is $\binom{R+N-1}{N-1}$, including the case $N=1$. Consequently
\[
 T_{\mathcal G}(\varepsilon)
 \le C_{\mathcal G}\sum_{R=0}^\infty
       (R+1)^{N-1}2^{d_*R}e^{-c\varepsilon2^R}
 =O\!\left(\varepsilon^{-d_*}
                  (1+|\log\varepsilon|)^N\right).
\]
Here is the cutoff calculation, also when $d_*=0$. For $0<\varepsilon\le1/2$ put $L=\lfloor\log_2(\varepsilon^{-1})\rfloor$. The sum for $R\le L$ is at most $(L+1)^N2^{d_*L}$. For $R=L+t$, $t\ge1$, use $\varepsilon2^L>1/2$; the remaining sum is at most a constant times
\[
 \varepsilon^{-d_*}(L+1)^{N-1}
 \sum_{t\ge1}(t+1)^{N-1}2^{d_*t}e^{-(c/2)2^t},
\]
whose final series converges. All constants are uniform on a compact positive rate set. The estimate is independent of an upper skeleton-label cutoff and proves finiteness of each discrete diagram sum before that cutoff is removed.

This argument covers arbitrary contact multiplicity without assuming summability at a rescaled zero boundary. In particular, if a low initial set $S$ has $L(S)<0$, its complement has $D(S^c)=D(\mathrm{all})-L(S)>D(\mathrm{all})$; that possible larger degree is already included in $d_*$. The scale gaps share the constraint $\sum_hd_h=R$, so their positive exponents are maximized, rather than added as independent powers of $\varepsilon^{-1}$.

Apply Lemma~\ref{ms:lem:general-marked-bounds}. In a main contact diagram $r=0$ and $C\ge1$, so $D(\mathrm{all})=B-2C\le B-2$, while every proper nonempty $H$ satisfies $D(H)\le B-3$. Hence $d_*\le\max(B-2,0)$. In a repeated-face diagram $r=1$, $C\ge0$, so $d_*\le\max(B-1,0)$. Finally, a noncubic main diagram has $B\le4k-12$, $D(\mathrm{all})\le B$ and $D(H)\le B-3$, so $d_*\le\max(B,0)$. There are finitely many diagrams for fixed $k$, and $N=v+I\le(2k-4)+(k+1)=3k-3$. Thus one may take the fixed logarithmic power $b=3k-3$ in the uniform bounds
\[
 \begin{array}{ll}
 \text{main contacts:}&
 O\!\left(\varepsilon^{-\max(B-2,0)}(1+|\log\varepsilon|)^b\right),\\
 \text{repeated-face diagrams:}&
 O\!\left(\varepsilon^{-\max(B-1,0)}(1+|\log\varepsilon|)^b\right),\\
 \text{noncubic main diagrams:}&
 O\!\left(\varepsilon^{-\max(B,0)}(1+|\log\varepsilon|)^b\right).
 \end{array}
\]
Write $A=4k-10\ge2$. Since always $B\le A$, the first two classes have exponents at most $A-2$ and $A-1$, respectively; the noncubic class has exponent at most $A-2$. Thus multiplication by $\varepsilon^A$ sends all three classes to zero locally uniformly, including schemes with one original vertex, where $B=0$.
\end{proof}

\subsection{The separated cubic diagrams}
We now consider the diagrams that remain. Separation means that marks occupy distinct internal chain vertices; their labels may still coincide. The positive subset exponents will turn the local kernel limits into a uniform limit of the full label sums.
\begin{Proposition}[Uniform Riemann-sum limit]\label{ms:lem:v8-separated-limit}
Fix a cubic scheme with $k\ge3$ numbered faces, one chosen corner mark of each face, and a fixed order of the marks along each scheme edge. Assume the marks are separated. Let $\mathcal G$ be the subdivided diagram, with vertex set $V$ of size $N=3k-4$ and edge multiset $\mathcal E$ of size $E=4k-6$. Write $m_i\in V$ for the vertex carrying the mark of face $i$. For each segment $e=uv$, let $i(e),j(e)$ be its two incident face labels, allowing $u=v$ and $i(e)=j(e)$.

For $\lambda\in(0,\infty)^k$ and sufficiently small $\varepsilon>0$, evaluate $G_{ij}$ and $\delta_i$ at the weights $g_i(\varepsilon)$ in \eqref{eq:critical-scaling}. For a label vector $l\in\{1,2,\ldots\}^{V}$, define
\[
 W_{\mathcal G,\varepsilon}(l;\lambda)
 =\prod_{e=uv\in\mathcal E}G_{i(e)j(e)}(l_u,l_v)
   \prod_{i=1}^k\delta_i(l_{m_i}),\qquad
 T_{\mathcal G}(\varepsilon;\lambda)
 =\sum_{l\in\{1,2,\ldots\}^{V}}W_{\mathcal G,\varepsilon}(l;\lambda).
\]
There is no distinct-label restriction: separation concerns the marked vertices. Each loop supplies its diagonal kernel factor, and parallel segments are separate factors. This is the sum for a fixed choice of marked corners and orders, before multiplication by the fixed inverse-automorphism weight; choices of different corners on the same face are counted separately.

At $\eta_i=\lambda_i/4$, define the nonnegative integral
\[
 J_{\mathcal G}(\lambda)
 =\int_{(0,\infty)^{V}}
   \prod_{e=uv\in\mathcal E}\bigl(3K_{i(e)j(e)}(x_u,x_v)\bigr)
   \prod_{i=1}^k\frac{p_i(x_{m_i})}{3}
   \prod_{v\in V}dx_v.
\]
Then $J_{\mathcal G}$ is finite and continuous, and the normalized discrete diagram sum converges to this integral:
\[
 \varepsilon^{4k-10}T_{\mathcal G}(\varepsilon;\lambda)
 \longrightarrow J_{\mathcal G}(\lambda)
\]
locally uniformly on positive rate sets.

More precisely, define $f_\varepsilon(x;\lambda)$ to be
$\varepsilon^{E-3k}W_{\mathcal G,\varepsilon}(l;\lambda)$ on each cell
$\prod_{v\in V}[\varepsilon l_v,\varepsilon(l_v+1))$, $l\in\{1,2,\ldots\}^{V}$,
and zero if any coordinate is below $\varepsilon$. For every compact
$Q\subset(0,\infty)^k$, there exist $c,\varepsilon_0>0$ such that, uniformly
for $\lambda\in Q$ and $0<\varepsilon<\varepsilon_0$,
\[
 \int_{\min_vx_v<\delta}f_\varepsilon(x;\lambda)\,\prod_vdx_v=O(\delta),
 \qquad
 \int_{\max_vx_v>T}f_\varepsilon(x;\lambda)\,\prod_vdx_v=O(e^{-cT})
 \quad(\delta,T>0).
\]
The same bounds hold for the total mass of cells on which some
$\varepsilon l_v<\delta$, or some $\varepsilon l_v>T$, respectively.
Here $c$, $\varepsilon_0$, and the implicit constants may depend on $Q$ and
the diagram, but not on $\lambda$, $\varepsilon$, $\delta$, or $T$.
Multiplying both the sum and the integral by a fixed diagram weight preserves
these conclusions.
\end{Proposition}
\begin{proof}
Use Remark~\ref{ms:lem:three-face-kernels} and Lemma~\ref{ms:lem:general-marked-bounds}. Fix a compact positive rate set $Q\subset(0,\infty)^k$ and choose a common sufficiently small $\varepsilon_0>0$ for the kernel and mark estimates on $Q$. Throughout the proof, $\lambda\in Q$ and $0<\varepsilon<\varepsilon_0$; constants may depend on $Q$ and the diagram. Consider a fixed cubic separated main diagram $\mathcal G$. Write $N=|V(\mathcal G)|$, $E=|E(\mathcal G)|$ and $M=k$, with edges counted with multiplicity, and let $\nu(v)$ be the number of marks at vertex $v$. Write $A=4k-10$. Here
\[
 N=3k-4,\qquad E=4k-6,\qquad N+E-3M=A=4k-10\ge2.
\]
For every proper nonempty vertex set $S$, each connected component of $S$ is proper and has $r_S=C_S=0$. The connected-set bound in Lemma~\ref{ms:lem:general-marked-bounds}, and additivity over these components, give $L(S)\ge3$. Also $L(V)=A$. We now connect these exponents explicitly to a common integrable envelope.

For lattice coordinates $y_v=\varepsilon l_v$, put a step function on the cells
\[
 \prod_v[\varepsilon l_v,\varepsilon(l_v+1)),\qquad l_v\ge1,
\]
by multiplying one factor $\varepsilon G_{ij}(l_u,l_v)$ per edge and one factor $\varepsilon^{-3}\delta_i(l_v)$ per mark. Set it to zero if any coordinate is below $\varepsilon$, and denote it by $f_\varepsilon(x;\lambda)$. Fixed scheme weights can be included in an overall constant. Its integral is exactly the normalized diagram sum:
\[
 \int_{(0,\infty)^N}f_\varepsilon(x;\lambda)\,dx
 =\varepsilon^{N+E-3M}T_{\mathcal G}(\varepsilon;\lambda)
 =\varepsilon^A T_{\mathcal G}(\varepsilon;\lambda).
\]
On every nonzero cell, $x_v/2\le y_v\le x_v$, since $l_v\ge1$. Thus the polynomial part of each rescaled kernel is at most four times
$\min(x_u,x_v)^3/\max(x_u,x_v)^2$, and each rescaled mark is at most $8x_v^{-3}$. Connectedness and the presence of a mark give $\max_v y_v\le\sum_{uv}|y_u-y_v|+\sum_v\nu(v)y_v$; this path estimate applies to the $y_v$ and gives exponential decay in $\max_vy_v\ge\frac12\max_vx_v$. Consequently there are constants $C,c_0>0$, uniform for $\lambda$ in any fixed compact positive rate set, such that
\[
 0\le f_\varepsilon(x;\lambda)\le C\Phi(x),\qquad
 \Phi(x)=
 \prod_{uv\in E(\mathcal G)}\frac{\min(x_u,x_v)^3}{\max(x_u,x_v)^2}
 \prod_vx_v^{-3\nu(v)}e^{-c_0\max_vx_v}.
\]
The envelope is independent of $\varepsilon$. Loops contribute $x_v$, multiple edges are repeated factors, and different retained vertices may have equal labels.

To prove integrability of $\Phi$, split each coordinate into dyadic intervals $[2^{n_v},2^{n_v+1})$ with $n_v\in\mathbb Z$. Order the block indices as $n_1\le\cdots\le n_N$, breaking ties arbitrarily and summing over at most $N!$ orders. Integrating the envelope over such a block is bounded by
\[
 C2^{P(n)}e^{-c_0 2^R},\qquad R=n_N,
\]
where the exponent $P$ defined by
\[
P(n)=\sum_v(1-3\nu(v))n_v+\sum_{uv}(3\min(n_u,n_v)-2\max(n_u,n_v))
\]
now has integer indices of either sign. Set
\[
 d_h=n_{h+1}-n_h\ge0,\qquad S_h=\{v_1,\ldots,v_h\},
 \qquad 1\le h<N.
\]
Starting with all block indices equal to $R$, lower the indices in $S_h$ by $d_h$. Lowering such a set contributes minus one from each coordinate volume, minus one from each internal edge, minus three from each crossing edge whose lower endpoint lies in $S_h$, and plus three for each mark in $S_h$. Hence the exact nested-gap identity is
\[
 P(n)=AR-\sum_{h=1}^{N-1}L(S_h)d_h.
\]
An internal loop contributes minus one just like any internal edge. This argument does not require $S_h$ to be connected. It gives the explicit summable bound
\[
 \int\Phi\le
 C\left(\sum_{R\in\mathbb Z}2^{AR}e^{-c_0 2^R}\right)
 \prod_{h=1}^{N-1}\left(\sum_{d_h\ge0}2^{-L(S_h)d_h}\right)<\infty
\]
for each ordering, followed by their finite sum. The $R\to-\infty$ tail converges because $A>0$, the $R\to+\infty$ tail by the exponential, and every gap sum because $L(S_h)>0$.

The same calculation supplies quantitative uniform boundary control. On a block,
\[
 (\min_vx_v)^{-1}\le2^{-n_1}
 =2^{-R+\sum_hd_h}.
\]
Thus $\int\Phi/\min_vx_v<\infty$: replace $A$ by $A-1>0$ and each $L(S_h)$ by $L(S_h)-1>0$ in the preceding geometric sums. It follows that, for $\delta>0$,
\[
 \int_{\min_vx_v<\delta}\Phi(x)\,dx\le C\delta.
\]
Replacing $c_0$ by $c_0/2$ in the same integrability calculation also gives
\[
 \int_{\max_vx_v>T}\Phi(x)\,dx\le C e^{-c_0T/2}.
\]
These estimates hold for the step functions as well, uniformly for $\lambda\in Q$ and $0<\varepsilon<\varepsilon_0$. In particular the cells with some $l_v=1$ lie in $\{\min x_v<2\varepsilon\}$ and have normalized total contribution $O(\varepsilon)$. More generally, if some $\varepsilon l_v<\delta$, their total is $O(\delta)$: it is empty for $\delta\le\varepsilon$, and otherwise their cells lie in $\{\min x_v<2\delta\}$. This controls the cutoff at label one rather than extrapolating a continuum estimate below it.

On every box $[\delta,T]^N$, Remark~\ref{ms:lem:three-face-kernels} makes $f_\varepsilon$ converge locally uniformly in the rates to the product of the limiting kernel and mark factors. This includes equal-coordinate pairs and loops, because its kernel convergence is locally uniform for all positive coordinate pairs. The common envelope and the two tail estimates permit first passing to the limit on this box and then sending $\delta\downarrow0$, $T\to\infty$, uniformly on compact positive rate sets. This proves the required dominated Riemann-sum passage and finiteness of the limiting diagram integrals. The same truncation proves continuity: the finite-cutoff lattice sums are continuous in the rates, their label tails converge locally uniformly by the exponential bound, and the normalized sums have the locally uniform limit just established.
\end{proof}

\subsection{Completion of the main theorem}
We combine the finite diagram expansion with the two preceding propositions. It remains to check that the corner multiplicities and the constants agree with the integral defined in the Introduction.
\begin{proof}[Proof of Theorem~\ref{ms:prop:general-analytic}]
Proposition~\ref{ms:lem:v8-marking-reduction} gives the exact positive marking expansion and bounds its replacement error by repeated-face diagrams. Proposition~\ref{ms:lem:v8-diagram-errors} makes that error, all contacts, and all noncubic main diagrams vanish after scaling by $\varepsilon^{4k-10}$. Proposition~\ref{ms:lem:v8-separated-limit} gives the locally uniform limit for every remaining cubic separated diagram, together with finiteness and continuity. The corner-tree factors at each original vertex have already been paired in the exact reduction, so no additional vertex weight remains. These estimates also justify numerical evaluation of the series. The error-diagram sums are finite for sufficiently small $\varepsilon$ by Proposition~\ref{ms:lem:v8-diagram-errors}, and the separated cubic sums are finite by the uniform boundary bounds in Proposition~\ref{ms:lem:v8-separated-limit} and the finite sum on each bounded label box. The positive marking expansion therefore bounds $\mathcal Z_k(g(\varepsilon))$ by a finite sum of finite quantities. This holds uniformly for the parameters in any fixed compact positive rate set. In fact $\mathcal Z_k(g)<\infty$ whenever $0<g_i<1/12$: choose a sufficiently small $\varepsilon>0$ for the rate vector $(1,\ldots,1)$ such that $g_i\le(1-\varepsilon^4/36)/12$ for every $i$, and compare each nonnegative map weight with its value at these equal, larger weights.

Finally sum over the finitely many cubic separated diagrams. There are $3k-6+k$ kernel segments and $k$ marks. Their constants multiply to
\[
 3^{3k-6+k}(1/3)^k=3^{3k-6}.
\]
To identify the surviving integrals, send each face label to the original scheme edge carrying its mark. This gives a map $\alpha\in\mathcal A_G$. At an internal vertex of a chain whose two sides belong to face $i$, the two incident corners of that face give the coefficient two in $F_\alpha$. When two distinct face marks lie on the same chain, their two orders give the two summands in the fourth case of $F_\alpha$. Regrouping by $\alpha$, with the original symmetry weights and parameters $\eta_i=\lambda_i/4$, therefore gives exactly $\mathcal I_{\{1,\ldots,k\}}(\lambda/4)$. This proves the claimed locally uniform limit.
\end{proof}

\section*{Acknowledgements}
The author acknowledges support from the National Key R\&D Program of China (Grant No.~2023YFA1010102) and the National Natural Science Foundation of China (Grant Nos.~12301173 and~12671171).

The author acknowledges the use of large language models (OpenAI's GPT-6/6.1), through an iterative prompting process, to support the conceptual development and proof construction in this paper. The author has independently verified all mathematical claims and arguments, and assumes full responsibility for the correctness and academic integrity of the final results.

\end{document}